\documentclass[11pt]{article}

\usepackage{amsmath,amssymb,amscd,amsthm,amsfonts}
\usepackage{graphicx,subfigure}
\usepackage{hyperref}
\usepackage{pstricks}
\usepackage{dsfont}
\usepackage{pifont}% http://ctan.org/pkg/pifont

\newtheorem{theorem}{Theorem}[section]
\newtheorem{proposition}[theorem]{Proposition}

\newtheorem*{theoremp}{Theorem}
\newtheorem{lemma}[theorem]{Lemma}

\newtheorem{conjecture}[theorem]{Conjecture}

\newtheorem{definition}[theorem]{Definition}

\newcommand{\N}{\mathbb{N}}
\newcommand{\R}{\mathds{R}}

\newcommand{\ff}{\mathcal{F}}

\newcommand{\cf}{{\cal F}}
\newcommand{\ck}{{\cal K}}

\newcommand{\ep}{\varepsilon}

\DeclareMathOperator{\conv}{conv}

\DeclareMathOperator{\vol}{vol}

\title{\bf Quantitative combinatorial geometry for continuous parameters
}

\author{J. A. De Loera \and R. N. La Haye \and D. Rolnick \and P. Sober\'on}

\begin{document}

\maketitle

\begin{abstract}
We prove variations of Carath\'eodory's, Helly's and Tverberg's theorems where the sets involved are measured according to continuous functions such as the volume or diameter.  Among our results, we present continuous quantitative versions of Lov\'asz's colorful Helly theorem, B\'ar\'any's colorful Carath\'eodory's theorem, and the colorful Tverberg theorem.
\end{abstract}

\section{Introduction}

Carath\'eodory's, Helly's, and Tverberg's theorems are undoubtedly among the most important theorems in convex geometry (see \cite{Mbook} for an introduction).  Many generalizations and extensions, including colorful, fractional, and topological versions of these theorems have been developed before. For a glimpse of the extensive literature see \cite{DLAS15, DGKsurvey63,Eckhoffsurvey93,Mbook,Wen1997,3nziegler} and the references therein. 

This paper presents several new quantitative versions of these theorems where we measure the size of the convex sets involved with a continuous function, such as the volume or diameter.
 
\begin{table}[h]\label{tablechida}
\begin{tabular}{|l|l|l|l|}
\hline \multicolumn{1}{|c|}{} &\multicolumn{1}{|c|}{Carath\'eodory} & \multicolumn{1}{|c|}{Helly}  &\multicolumn{1}{|c|}{Tverberg}  \\
\hline  Monochromatic version & \checkmark+ $(2.4-5)$& \checkmark+ $(3.1, 3.5)$ & $(4.1)$ \\
\hline  Colorful version & $( 2.4-6)$ & $(3.6)$ & $(4.2)$  \\
\hline
\end{tabular}\vspace{-1pt}

\caption{\emph{Prior and new results in quantitative combinatorial convexity.} The symbol $\checkmark $ means some prior result was known, $(\#)$ indicates the number of the theorem that is a brand new result or a stronger version of prior results.}\vspace{0pt}
\end{table}

The equivalent results for discrete parameters are discussed in the paper \cite{discrete-new}, where the focus is on Tverberg-type results.  The full picture is presented in \cite{quantitativefull}, which was split into this paper and \cite{discrete-new} for publication.

Let us recall the statement of the original results.

\begin{theoremp}[Carath\'eodory, 1911 \cite{originalCaratheodory}] 
\label{thm:Caratheodory}
Let $S$ be any subset of $\R^d$. Then each point in the convex hull of $S$ is a convex combination of at most $d+1$ points of $S$.
\end{theoremp}

\begin{theoremp}[Helly, 1913 \cite{originalHelly}] 
\label{thm:Helly}
Let $\cf$ be a finite family of convex sets of $\R^d$. If $\bigcap \ck \neq
\emptyset$ for all $\ck \subset \cf$ of cardinality at most $d+1$, then
$\bigcap \cf \neq \emptyset$. 
\end{theoremp}

\begin{theoremp}[Tverberg, 1966 \cite{Tverberg:1966tb}]
\label{thm:Tverberg}
 Let $a_{1},\ldots,a_{n}$ be points in $\R^{d}$.
If the number of points satisfies $n >(d+1)(m-1)$, then they can be partitioned into $m$ disjoint parts $A_{1},\ldots,A_{m}$ in such
a way that the $m$ convex hulls $\conv A_1, \ldots, \conv A_m$ have a point in common.
\end{theoremp}

The case of $m=2$ in Tverberg's theorem was proved in 1921 by Radon \cite{originalRadon} and is often referred to as Radon's theorem or Radon's lemma. 

A key idea in many of our proofs is to link these three classical theorems with the efficient approximation of convex sets by polytopes. 
We were thus able to apply recent developments in the active field of convex body approximation (for which we recommend the surveys \cite{bronstein2008approximation, gruber1993aspects}). % It is worth stressing our techniques extend to prove theorems about other continuous measurable invariants; for example the surface area. 
In a separate paper we will discuss analogous results where one counts points in discrete sets.

In the remainder of the introduction, we discuss prior work and describe our results divided by the type of theorem. In Section \ref{section-caratheodory} we prove our Carath\'eodory-type results, in Section \ref{section-helly} the Helly-type results, and finally in Section \ref{section-Tverberg} the Tverberg-type results.

\subsection*{Carath\'eodory-type contributions}

Carath\'eodory's theorem has interesting consequences and extensions (e.g., \cite{baranyonn-colorfulLP,Mbook}).
In 1914, Steinitz improved the original proof by Carath\'eodory (which applied only to compact sets \cite{originalCaratheodory}) and at the same time he was the first to describe a version of this theorem for points in the interior of a convex set:

\begin{theoremp}[Steinitz, 1914 \cite{originalsteinitz}]
Consider $X \subset \R^d$ and $x$ a point in the interior 
of the convex hull of $X$. Then, $x$ belongs to the interior of the convex hull of a set of at most $2d$ points of $X$.
\end{theoremp}

More generally, a \emph{Carath\'eodory-type theorem} describes points within the convex hull of a set $S$ as convex combinations of a given number of generators, under some additional conditions.  One important generalization of Carath\'eodory's theorem, commonly known as the colorful Carath\'eodory theorem, was given by B\'ar\'any:

\begin{theoremp}[B\'ar\'any, 1982 \cite{baranys-caratheodory}]
Given $d+1$ subsets $X_1, X_2, \ldots, X_{d+1}$ in $\R^d$ and a point $x$ such that $x \in \conv X_i$ for all $i$, we can find points $x_1 \in X_1, \ldots, x_{d+1}	 \in X_{d+1}$ such that $x \in \conv \{x_1, \ldots , x_{d+1}\}$.
\end{theoremp}

This result is called \emph{colorful} since we can consider each $X_i$ as colored with a different color. By contrast, a non-colorful version will be called \emph{monochromatic}. Note that the case $X_1 = \ldots = X_{d+1}$ above gives us the original, monochromatic result.  

B\'ar\'any, Katchalski, and Pach were the first to present quantitative theorems in combinatorial convexity, including a monochromatic quantitative Carath\'eodory-type theorem. We denote by $B_r(p) \subset \R^d$ the Euclidean ball of radius $r$ with center $p$.

\begin{theoremp}[Quantitative Steinitz theorem, B\'ar\'any, Katchalski, Pach 1982 \cite{baranykatchalskipach, MR750523}]
There is a constant $r(d) \ge d^{-2d}$ such that the following statement holds.  For any set $X$ such that $B_1(0) \subset \conv X$, there is a subset $X' \subset X$ of at most $2d$ points that satisfies $B_{r(d)}(0) \subset \conv X'$.
\end{theoremp}

B\'ar\'any et al.~used this theorem  as a key lemma to prove their main quantitative results. We prove that the result above admits a colorful version, Lemma \ref{theorem-colored-steinitz}.  A colorful version of Steinitz' original (non-quantitative) theorem was also noted by Jer\'onimo-Castro but has not been published \cite{jeronimo-colorful-communication}.  

Our key contribution is a version of the result above where we seek to optimize the volume of the resulting set, instead of the number of vertices forming $X'$.  This is shown in Theorem \ref{theorem-thrifty-steinitz}, which provides tight asymptotic bounds for the number of sets involved.  In particular, the conclusion gives $\operatorname{vol} (\conv X') \ge (1-\ep)\operatorname{\vol(\conv X)}$, using the notation above, for a positive $\ep$ fixed in advance.  The size of the subset obtained is closely related to the number of vertices needed for an inscribed polytope to efficiently approximate the volume of a convex set.

For the applications of colorful Steinitz theorems, we need to work with slightly different parameters than the volume. We show a variation of this type with Theorem \ref{corollary-steinitz-for-balls}, which is based on a different type of approximation of convex sets by polytopes.  The monochromatic version of Theorem \ref{corollary-steinitz-for-balls} was obtained previously in \cite{KMY92}.  Our methods shows which type of polytope approximation results yield quantitative Steinitz theorems.  These quantitative versions of Steinitz' theorem are at the core of our proofs for the quantitative versions of Helly's and Tverberg's theorems.

\subsection*{Helly-type contributions}

Helly's theorem and its numerous extensions are of central  importance in discrete and computational geometry
(see \cite{DLAS15, DGKsurvey63,Eckhoffsurvey93,Wen1997}). Helly himself understood immediately that his theorem had many variations, and was, for instance, the first to prove a topological version \cite{hellyagain}.  A \emph{Helly-type} property $P$ is a property for which there is a number $\mu$ such that the following statement holds.  \textit{If $\ff$ is a finite family of objects such that every subfamily with $\mu$ elements satisfies $P$, then $\ff$ satisfies $P$}.  In rough terms, we may summarize some of our results below as the statement that ``\textit{the intersection has a large volume}'' is a Helly-type property for convex sets.

To our knowledge, the first family of quantitative Helly-type theorems was made explicit by B\'ar\'any, Katchalski, and Pach in  \cite{baranykatchalskipach, MR750523}, who obtained extensions of the classic Helly and Steinitz theorems for convex sets with a volumetric constraint.

\begin{theoremp}[B\'ar\'any, Katchalski, Pach, 1982 \cite{baranykatchalskipach, MR750523}]
Let $\ff$ be a finite family of convex sets such that for any subfamily $\ff'$ of at most $2d$ sets,
\[
\vol\left( \bigcap \ff'\right) \ge 1 .
\]
Then,
\[
\vol\left( \bigcap \ff \right) \ge d^{-2d^2}.
\]
\end{theoremp}

This result has recently been improved by Nasz\'odi to conclude $\vol (\cap \ff) \ge d^{-cd}$ for some absolute constant $c$ \cite{nazo15}.  The size of the subfamilies one must check cannot be improved over $2d$, as is noted in \cite{baranykatchalskipach}. In order to see this, let $\ff$ be the family of $2d$ half-spaces defining the facets of an arbitrarily small hypercube.  Any $2d-1$ define an unbounded polyhedron with non-empty interior, showing the optimality of their result.

In Section \ref{section-helly}, we analyze the other side of the spectrum. We present Theorem \ref{theorem-helly-for-volume}, showing that it is possible to obtain better approximations of the volume of the intersection, namely $\vol\left( \cap \ff \right) \ge 1-\ep$ for some $\ep>0$ fixed in advance, if one is willing to check for subfamilies $\ff'$ of larger size. This answers a question raised by Kalai and Linial during an Oberwolfach meeting in February 2015.  Moreover, we show that the loss of volume $\ep$ is unavoidable.  In other words, it is impossible to conclude $\vol(\cap \ff) \ge 1$ regardless of the size of the subfamilies we are willing to check.  This is a remarkable difference between the continuous and discrete quantitative Helly-type theorems (see \cite{alievetal} for the existing discrete counterpart).  The bounds we present for Theorem \ref{theorem-helly-for-volume} are optimal in their dependence on $\ep$.

Similar Helly-type quantitative results were known previously only with other types of functions, related to the expansion and contraction of convex sets \cite{LS09}.  The proof method extends naturally to other functions as long as a polytope approximation result is proved.  We show how this method gives an analogous result for the diameter, Theorem \ref{theorem-helly-for-diameter}.
%\begin{theorem}[Quantitative Helly for Volume and Diameter]\label{super-quant-helly}
%Let $d\in\N$, $\ep>0$, and take either 
%\begin{enumerate}
%\item[(a)] $n=n^\star(d,\ep)$ as in Definition \ref{definition-for-helly} and $M(K)=\vol(K)$, or 
%\item[(b)] $n=n^{\operatorname{diam}}(d,\ep)$ as in Proposition \ref{definition-diameter} and $M(K)=\operatorname{diam}(K)$.
%\end{enumerate}
%Let $\ff$ be a finite family of convex sets such that for any subfamily $\ff'$ of at most $nd$ sets, 
%\[
%M\left( \bigcap \ff'\right) \ge 1	.
%\]
%Then,
%\[
%M\left( \bigcap \ff \right) \ge (1+\ep)^{-1}.
%\]
%Moreover, $n$ is a lower bound for the size of the subfamilies $\ff'$ that we need to check.
%\end{theorem}
%The lower bounds presented in Theorem \ref{super-quant-helly} show that it is impossible to conclude $\vol(\cap \ff) \ge 1$ or $\operatorname{diam}(\cap\ff)\ge 1$, respectively, regardless of the size of the subfamilies we are willing to check.  This is a remarkable difference between the continuous and discrete quantitative Helly-type theorems (see \cite{alievetal} for the existing discrete counterpart).  

We also extend the volumetric result to the colorful setting, in Theorem \ref{thm:hcolorcont}, in the same spirit as Lov\'asz's generalization of Helly's theorem \cite{baranys-caratheodory}.  However, since the estimates for the size of the subfamilies needed to check vary greatly among the two results, we state them separately.

Quantitative versions of newer results regarding intersection structure of families of convex sets are shown by Rolnick and Sober\'on in \cite{Sob15}.  These are closely related to the contributions of this paper and include versions of the $(p,q)$ theorem of Alon and Kleitman \cite{Alon92pq} and the fractional Helly theorem of Katchalski and Liu \cite{Kat79frac}.

\subsection*{Tverberg-type contributions}

Tverberg published his classic theorem in 1966  \cite{Tverberg:1966tb}. Later in 1981 he published another proof  \cite{Tverberg1981}, and simpler proofs have since appeared in  \cite{baranyonn-colorfulLP, Roudneff, Sarkaria:1992vt}.  Section 8.3 of \cite{Mbook} and the expository article \cite{3nziegler} can give the reader a sense of the abundance of work surrounding this elegant theorem. Here we present the first quantitative versions with continuous parameters.

First,  we prove a version of Tverberg's theorem, Theorem \ref{theorem-tverberg-volume}, where each convex hull must contain a Euclidean ball of given radius.  In other words, we measure the ``size'' of $\cap_{i=1}^m \conv A_i$ by the inradius. Our proof combines Tverberg's theorem with our two versions of quantitative Steinitz' theorem for balls, Lemma \ref{theorem-colored-steinitz} and Proposition \ref{corollary-steinitz-for-balls}. %The constant $n^{\operatorname{bm}}(d,\ep) \sim e^{-d/2}$ will be made explicit in Definition \ref{banach-mazur-value}. 

Note that, unlike the classical Tverberg theorem, some conditions must be imposed on the set of points to be able to obtain such a result.  For instance, regardless of how many points we start with, if they are all close enough to some flat of positive co-dimension, then all hopes of a continuous quantitative version of Tverberg's theorem quickly vanish.  In order to avoid the degenerate cases, we make the natural assumption that the set of points is ``thick enough''.

As with Helly's and Carath\'eodory's theorems, there are colorful versions of Tverberg's theorem.  In this case, the aim is to impose additional combinatorial conditions on the resulting partition of points, while guaranteeing the existence of a partition where the convex hulls of the parts intersect. Now that the conjectured topological versions of Tverberg's theorem have been proven false \cite{frick15} (using the generalized Van Kampen--Shapiro--Wu theorem \cite{Mabillard2014}), the following conjecture by B\'ar\'any and Larman is arguably the most important open problem surrounding Tverberg's theorem.

\begin{conjecture}[B\'ar\'any and Larman 1992 \cite{Barany:1992tx}]\label{conjecture-colorful-Tverberg}
Let $F_1, F_2, \ldots, F_{d+1} \subset \R^d$ be sets of $m$ points each, considered as color classes.  Then, there is a colorful partition of them into sets $A_1, \ldots, A_{m}$ whose convex hulls intersect.
\end{conjecture}

By a colorful partition $A_1, \ldots, A_m$ we mean that it satisfies $|A_i \cap F_j| = 1$ for all $i,j$. In presenting the conjecture, B\'ar\'any and Larman showed that it holds for $d=2$ and any $m$, and included a proof by L\'ov\'asz for $m=2$ and any $d$.  Recently, Blagojevi\'c, Matschke, and Ziegler \cite{blago3, bmz-optimal} showed that it is also true for the case when $m+1$ is a prime number and any $d$. The reason for these conditions on the parameters of the problem is that their method of proof uses topological machinery requiring these assumptions. However, their result shows that if we allow each $F_i$ to have $2m-1$ points instead of $m$, we can find $m$ pairwise disjoint colorful sets whose convex hulls intersect, without any conditions on $m$.  For variations of Conjecture \ref{conjecture-colorful-Tverberg} that do imply Tverberg's theorem, see \cite{blago3, bmz-optimal, Soberon:2013fr}.

%We stress that the colorful versions of Tverberg's theorem are in essence completely different from the colorful versions of Helly and Carath\'eodory.  With Helly and Carath\'eodory, when the color classes are equal we get the original (monochromatic) version.  If the color classes of the conjecture above are equal, we do not get Tverberg's theorem back. Actually, the variations of Conjecture \ref{conjecture-colorful-Tverberg} that would imply Tverberg's theorem do so only when the color classes are ``very different''; see \cite{blago3, bmz-optimal} and \cite{Soberon:2013fr} for two such versions.

Combining results of Blagojevi\'c, Matschke, and Ziegler with our two colorful Steinitz theorems, we also obtain volumetric versions of these results in Theorem \ref{theorem-colorful-tverberg-volume}.  In order to obtain a ball in the intersection, for these results we must allow each $A_i$ to have more points of each color class.  It should be stressed that other interpretations of quantitative Tverberg's theorem are possible, and some are presented in \cite{Sob15}.

\section{Quantitative Carath\'eodory theorems}\label{section-caratheodory}

We prove only the colorful versions of our Carath\'eodory-type theorems.  Given sets $X_1, \ldots, X_n$, considered as color classes, whose convex hulls contain a large set $K$, we want to make a colorful choice $x_1 \in X_1, \ldots, x_n \in X_n$ such that $\conv \{x_1, \ldots, x_n\}$ is also large. The monochromatic versions of the results below follow from the case $X_1 = X_2 = \cdots = X_n$.

There are two parameters we may seek to optimize.  One is the number $n$ of sets required to obtain some lower bound for the size of $\conv \{x_1, \ldots, x_n\}$.  The other is the size of $\conv \{x_1, \ldots, x_n\}$ assuming that the size of $K$ is $1$. These considerations lead to different results.

The first result of this kind is a monochromatic quantitative version of Steinitz' theorem by B\'ar\'any, Katchalski, and Pach \cite{baranykatchalskipach, MR750523}, quantifying the largest size of a ball centered at the origin and contained in $K$, described in the introduction. The case when $X$ is the set of vertices of a regular octahedron centered at the origin shows that the number of points they use, $2d$, cannot be reduced. We will later show how adapting the proof of \cite{baranykatchalskipach} to the colorful case gives Lemma \ref{theorem-colored-steinitz}.

On the other side, it is natural to optimize the size of $\conv\{x_1, \ldots, x_n\}$ instead of the integer $n$.  This optimization turns out to be closely related to finding efficient approximations of convex sets with polytopes.  This is a classic problem which has many other motivations, see \cite{barvinok2014thrifty, bronstein2008approximation, gruber1993aspects} for the state of the art and the history of this subject. In this paper we will need the following constants.

\begin{definition}\label{definition-inscribed-volume}
	Let $d$ be a positive integer and $\ep>0$.  We define $n(d,\ep)$ as the smallest integer such that, for any convex set $K \subset \R^d$ with positive volume, there is a polytope $P \subset K$ of at most $n(d,\ep)$ vertices such that
	\[
	\vol(P) \ge (1-\ep) \vol(K).
	\]
\end{definition}

\begin{definition}\label{definition-for-helly}
Let $d$ be a positive integer and $\ep>0$.  We define $n^*(d,\ep)$ as the smallest integer such that, for any convex set $K \subset \R^d$ with positive volume, there is a polytope $P \supset K$ of at most $n^*(d,\ep)$ facets such that
\[
\vol(P) \le (1+\ep) \vol(K).
\]	
\end{definition}

%\begin{definition}\label{banach-mazur-value}
%	Let $d$ be a positive integer and $\ep>0$.  We define $n^{\operatorname{bm}}(d,\ep)$ as the smallest integer such that, for any centrally symmetric convex set $K \subset \R^d$ with positive volume, there is a polytope $P$ of at most $n^{\operatorname{bm}}(d,\ep)$ vertices and a linear transformation $\lambda: \R^d \to \R^d$ such that
%	\[
%	P \subset \lambda(K) \subset (1+\ep )P.
%	\]
%	In other words, there is always a polytope $P \subset K$ of fixed number of vertices which is within $\varepsilon$ of $K$ according to the Banach-Mazur distance.
%\end{definition}

%Since there are better estimates for the quantity above if the set $K$ has a smooth enough boundary, it is convenient to include an additional definition.

\begin{definition}\label{banach-mazur-value-smooth}
	Let $d$ be a positive integer and $\ep>0$.  We define $n^{\operatorname{bms}}(d,\ep)$ as the smallest integer such that, for any centrally symmetric convex set $K \subset \R^d$ with positive volume and $C^2$ boundary, there is a polytope $P$ of at most $n^{\operatorname{bms}}(d,\ep)$ vertices and a linear transformation $\lambda: \R^d \to \R^d$ such that
	\[
	P \subset \lambda(K) \subset (1+\ep )P.
	\]
\end{definition}

The asymptotic behavior of $n(d, \ep)$ is known: 
\[\left(\frac{c_1d}{\ep}\right)^{(d-1)/2} \ge n(d,\ep) \ge \left(\frac{c_2d}{\ep}\right)^{(d-1)/2},\]
for absolute constants $c_1, c_2$. This comes from approximating convex bodies with polytopes of few vertices via the Nikodym metric \cite[Section 4.2]{bronstein2008approximation}.  The lower and upper bounds can be found in \cite{gordon1995constructing} and \cite{gordon1997umbrellas}, respectively.

We will use these definitions to obtain both upper and lower bounds for our quantitative results. As shown below, $n(d,\ep)$ is precisely the number needed for a quantitative colorful Steinitz theorem with volume.

The constant $n^{\operatorname{bms}}(d,\ep)$ will be needed to improve the quantitative Steinitz theorem if we are interested in determining the size of a set by the radius of the largest ball around the origin contained in it.  The bounds for $n^{\operatorname{bms}}(d, \ep)$ involve the condition of central symmetry as the Banach-Mazur distance is most natural when working with norms in Banach spaces.  There is a precise asymptotic bound $n^{\operatorname{bms}}(d,\ep) \le	 \left( \gamma {\ep}\right)^{-(d-1)/2}$, where $\gamma$ is an absolute constant \cite{Bor00,Gruber-other}.  The dependence on $\ep$ is optimal, as shown in \cite{blaschke1923affine}.  If the smoothness condition is removed from Definition \ref{banach-mazur-value-smooth}, the best known upper bound is given by Barvinok \cite{barvinok2014thrifty}, giving a bound of $O\left(\left({\ep^{-1/2}}\ln \left(\frac{1}{\ep}\right)\right)^d\right)$ if $d$ is large enough.
 
Finally, the constant $n^*(d,\ep)$ is the key value for the continuous quantitative Helly theorems in Section \ref{section-helly}.  A result of Reisner, Sch\"utt and Werner shows that $n^*(d, \ep) \le \left( \frac{c_1 d^2}{\varepsilon}\right)^{(d-1)/2}$ for an absolute constant $c_1$, and noted that this is optimal in the dependence on $\varepsilon$ \cite[Section 5]{reisner2001dropping}.

The only extra ingredient needed is the following result. For other extensions of the colorful Carath\'eodory theorem, see \cite{Arocha:2009ft,Holmsen:2008vl}.

\begin{theoremp}[Very colorful Carath\'eodory theorem, B\'ar\'any, 1982 {\cite[Theorem 2.3]{baranys-caratheodory}}] \label{verycolorfulcaratheodory}
Let $X_1, X_2, \ldots, X_d \subset \R^d$ be sets, each of whose convex hulls contains $p\in \R^d$ and let $q \in \R^d$.  Then, we can choose $x_1 \in X_1, \ldots, x_d \in X_d$ such that
	\[
	p \in \conv\{x_1, x_2, \ldots, x_d, q\}.
	\] 
\end{theoremp}

%Definition \ref{definition-inscribed-volume} gives the explicit value of $n(d,\ep)$, which gives the correct bound for the following result up to a multiplicative factor of $d$.

%\begin{theorem}[Colorful quantitative Steinitz with volume]\label{theorem-thrifty-steinitz}	For $d$ a positive integer and $\ep>0$ a constant, take $n=n(d,\ep)$ as in Definition \ref{definition-inscribed-volume}.  Then, the following property holds: If $X_1, X_2, \ldots, X_{nd}$ are sets in $\R^d$ and $K \subset \bigcap_{i=1}^{nd} \conv(X_i)$ is a convex set of volume $1$, we can choose $x_1 \in X_1, x_2 \in X_2, \ldots, x_{nd} \in X_{nd}$ so that
%\[
%\vol(\conv\{x_1, x_2, \ldots, x_{nd}\})\ge 1-\ep.
%\]
%Moreover, $n(d,\ep)$ is also a lower bound for the number of sets needed in this theorem.
%\end{theorem}

%\begin{theorem}[Quantified colorful Steinitz for volume]\label{theorem-thrifty-steinitz}	Let $n=n(d,\ep)$ and $X_1, X_2, \ldots, X_{nd}$ be sets in $\R^d$ and $K \subset \R^d$ a centrally symmetric convex set of volume $1$ such that $K \subset X_i$ for all $i$.  Then, we can choose $x_1 \in X_1, x_2 \in X_2, \ldots, x_{nd} \in X_{nd}$ so that
%\[
%\vol(\conv\{x_1, x_2, \ldots, x_{nd}\})\ge 1-\ep.
%\]
%Moreover, If we change the number of sets $nd$ by $n-1$, the conclusion may not hold.
%\end{theorem}

\begin{theorem}[Colorful quantitative Steinitz with volume]\label{theorem-thrifty-steinitz}	Let $d \in \N$ and $\ep>0$, consider $m=d \cdot n(d,\ep)$ as in Definition \ref{definition-inscribed-volume}.  Then, the following property holds: If $X_1, X_2, \ldots, X_{m}$ are sets in $\R^d$ and $K \subset \bigcap_{i=1}^{m} \conv(X_i)$ is a convex set of volume $1$, we can choose $x_1 \in X_1, x_2 \in X_2, \ldots, x_{m} \in X_{m}$ so that
\[
\vol(\conv\{x_1, x_2, \ldots, x_{m}\})\ge 1-\ep.
\]
Moreover, if $m<n(d,\ep)$ the conclusion of the theorem may fail.
\end{theorem}

\begin{proof}[Proof of Theorem \ref{theorem-thrifty-steinitz}]
	Let $P \subset K$ be a polytope with $n=n(d,\ep)$ vertices such that $\vol(P) \ge (1- \ep) \vol(K)=(1-\ep)$.  We may assume without loss of generality that $0$ is in the interior of $P$.  Now label the vertices of $P$ as $y_1, y_2, \ldots, y_n$.  Using the very colorful Carath\'eodory theorem, for a fixed $j$ we can find $x_{(j-1)d+1}{\in}X_{(j-1)d+1}$, $\ldots ,$ $x_{jd} \in X_{jd}$ such that
	\[
	y_j \in \conv\{0, x_{(j-1)d+1}, \ldots, x_{jd} \}.
	\]
	In order to finish the proof, it suffices to show that $0 \in \conv\{x_1, \ldots, x_{nd} \}$.  If this is not the case, then there is a hyperplane separating $0$ from $\conv\{x_1, \ldots, x_{nd} \}$.  We may assume that the hyperplane contains $0$ and leaves $x_1, \ldots, x_{nd}$ in the same closed half-space.  Notice that then there would be a vertex of $y_j$ of $P$ in the other (open) half-space, contradicting the fact that $y_j{\in}\conv\{0,x_1,\ldots,x_{nd}\}$. 
	
We now prove the near-optimality of our bound. Let $K$ be a convex set of volume $1$ such that for every polytope $P \subset K$ of at most $n-1$ vertices we have $\vol(P) < 1-\ep$.  Then, taking $K = X_1 = X_2 = \cdots = X_{n-1}$ gives the desired counterexample, as any colorful choice of points has size $n-1$.
\end{proof}

%If we want the subset to be close to $K$ in terms of the Banach-Mazur distance, 
%For part (b) of Theorem \ref{colorful-quant-steinitz}, we simply replace $n(d,\ep)$ by $n^{\operatorname{bm}}(d,\ep)$ and the same proof holds.

%\begin{theorem}[Quantified colorful Steinitz for Banach-Mazur distance]\label{theorem-steinitz-banach-mazur}
%Let $d$ be a positive integer and $\ep>0$ be a constant.  Set $n=n^{\operatorname{bm}}(d,\ep)$ and let $X_1, X_2, \ldots, X_{nd}$ be sets in $\R^d$ such that $K \subset \bigcap_{i=1}^{nd}\conv X_i$ is a centrally symmetric convex set with volume $1$.  Then, we can find $x_1 \in X_1, x_2 \in X_2, \ldots, x_{nd} \in X_{nd}$ and an affine transformation $\lambda:\R^d \to \R^d$ so that $\conv\{x_1, x_2, \ldots, x_{nd}\}$ contains a set $P$ with
%\[
%P \subset \lambda(K) \subset (1+\ep)P.
%\]
%Moreover, $n$ is a lower bound for the number of sets needed for this result to hold.
%\end{theorem}
 
Additionally, with the same ideas we get the following proposition, which improves the quantitative version of B\'ar\'any, Katchalski, and Pach when we want to optimize the radius of the balls contained in the set.  The number of sets we use is slightly improved by using the symmetries of the sphere.  The estimates are very similar to those of \cite{KMY92}, where the monochromatic version of the problem below was studied.  Indeed, the asymptotic growth in terms of $\ep$ is the same, and it is multiplied by a factor which is exponential in the dimension.  We did not chase the constants in the base of this exponential factor in the bounds for $n^{\operatorname{bms}}(d,\ep)$ to determine which result is best, but the proof with these methods is more concise and gives the colorful version.

\begin{theorem}\label{corollary-steinitz-for-balls}
Set $n=n^{\operatorname{bms}}(d,\ep)$ and let $X_1, X_2, \ldots, X_{(n-1)d+1}$ be sets in $\R^d$ such that $B_1 (0) \subset \bigcap_{i=1}^{(n-1)d+1} \conv X_i$. Then, we can choose $x_1 \in X_1, x_2 \in X_2, \ldots, x_{(n-1)d+1} \in X_{(n-1)d+1}$ so that
\[
B_{1/(1+\ep)}(0) \subset \conv\{x_1, x_2, \ldots, x_{(n-1)d+1}\}.
\]
\end{theorem}

\begin{proof}
We follow a proof similar to that of Theorem \ref{theorem-thrifty-steinitz}.
	Let $P \subset B_1(0)$ be a polytope with $n=n^{\operatorname{bms}}(d,\ep)$ vertices such that $B_1(0) \subset (1+\ep)P$.  Note that $0$ is in the interior of $P$.  Now label the vertices of $P$ as $y_1, y_2, \ldots, y_n$.  Using a rotation on $P$, we may assume that there is a point $x_{(n-1)d+1} \in X_{(n-1)d+1}$ such that $y_n \in \conv \{0, x_{(n-1)d+1}\}$. Using the very colorful Carath\'eodory theorem as before, for every $1\le j\le n-1$ we can find $x_{(j-1)d+1}{\in}X_{(j-1)d+1}$, $\ldots ,$ $x_{jd} \in X_{jd}$ such that
	\[
	y_j \in \conv\{0, x_{(j-1)d+1}, \ldots, x_{jd} \}.
	\]
	As in the proof of Theorem \ref{theorem-thrifty-steinitz}, we have that $0 \in \conv\{x_1, \ldots, x_{(n-1)d+1} \}$, yielding the desired result. 
\end{proof}

The reader may notice that an analogous proof of Theorem \ref{theorem-thrifty-steinitz} but using Definition \ref{banach-mazur-value-smooth} gives a version of the result above involving the Banach-Mazur distance to some fixed $K$ contained in each set of the form $\conv X_i$.  We finish the section by showing that the quantitative Steinitz theorem from B\'ar\'any, Katchalski and Pach mentioned in the introduction can be colored:

\begin{lemma}[Colorful quantitative Steinitz with containment of small balls]\label{theorem-colored-steinitz}
Let $r(d) \le (\pi/e^2)d^{-2d-2}$ and $X_1, X_2, \ldots, X_{2d}$ be sets in $\R^d$ such that $B_1 (0) \subset \conv(X_i)$ for all $i$.  Then, we can choose $x_1 \in X_1, x_2 \in X_2, \ldots, x_{2d} \in X_{2d}$ so that
\[
B_{r(d)}(0) \subset \conv\{x_1, x_2, \ldots, x_{2d}\}.
\]
\end{lemma}

%The reason for this result to be called ``colorful'' is that it has the following interpretation.  If every $X_i$ is painted with a different color, the theorem states that \textit{if the convex hull of every monochromatic set contains $B_1(0)$, then there is a colorful set whose convex hull contains $B_{r(d)}(0)$}.  This follows the lines of B\'ar\'any's generalization of Carath\'eodory's theorem \cite{baranys-caratheodory}: \textit{If $V_1,\cdots,V_{d+1}\subseteq{\R}^d$ and $p\in\bigcap_{i=i}^{d+1}\conv (\,V_i)$, then there exist elements $v_i\in V_i$, $1\leq i\leq d+1$, such that $p\in\conv\{v_1,\cdots,v_{d+1}\}$}.

\begin{proof}[Proof of Lemma \ref{theorem-colored-steinitz}] Our goal is to pick explicitly the $2d$ points $x_1,\dots,x_{2d}$. For this, let
 $P$ be a regular simplex of maximal volume contained in $B_1(0)$.  Note that $B_{1/d}(0) \subset P \subset B_1(0)$.  Since $P \subset \conv X_i$ for an arbitrary $i$ and $P$ has $d+1$ vertices, by repeatedly applying Carath\'eodory's theorem we can see that there is a subset of $X_i$ of size at most $(d+1)^2$ whose convex hull contains $P$.  Thus, without loss of generality we may assume $|X_i| \le (d+1)^2$ and $B_{1/d}(0) \subset \conv(X_i)$ for all $i$.
	
  Given a collection of $d$ points, $x_1 \in X_1$, $x_2 \in X_2$, $\ldots$, $x_d \in X_d$, consider the convex (simplicial) cone spanned by them.  Let $C_1, C_2, \ldots, C_n$ be all possible cones generated this way.  The number of cones, $n$, is clearly bounded by
  \[
  n \le (d+1)^{2d}.
  \]
  \noindent \textbf{Claim.} The cones $C_1, C_2, \ldots, C_n$ cover $\R^d$.
  
  In order to prove the claim, it suffices to show that for each vector $v$ of norm at most $\frac1d$, there is a cone $C_i$ that contains it.  However, since $B_{1/d}(0) \subset X_i$ for all $i$ (in particular for the first $d$), we can apply the very colorful Carath\'eodory theorem above with the point $p$ in the convex hull being $v$ and the extra point $q$ being $0$.
  
  If we denote by $\omega_{d-1}$ the surface area of the unit sphere $S^{d-1}$, there must be one of the cones 
  $C_i$ which covers a surface area of at least $\frac1n \omega_{d-1}$.  We can assume without loss of generality that it is the first cone $C_1$.
 
  Let $a \in C_1$ be a unit vector whose minimal angle $\alpha$ with the facets of $C_1$ is maximal (i.e. we take the incenter of $C_1 \cap S^{d-1}$, with distance measured in the sphere).  Now we show that since the surface area of $C_1 \cap S^{d-1}$ is large, its inradius must also be large.  The argument we present is different from \cite{baranykatchalskipach}, giving a slightly worse constant.  Our final radius is $(\pi/e^2)d^{-2d-2}$ as opposed to their $d^{-2d}$.
  
  For a facet $L_i$ of $C_1$, let $D_i$ be the set of points whose angle with $L_i$ is at most $\alpha$ and that lie on the same side of $L_i$ as $a$.  Note that $C_1$ has $d$ facets and 
  so $\cup_{i=1}^d D_i = C_1$.  The surface area of $S^{d-1} \cap D_i$ is clearly bounded by $\frac{\alpha}{2\pi}\omega_{d-1}$.  Thus
  \[
  \frac1n \omega_{d-1} \le \operatorname{Area}(S^{d-1} \cap C_1)  < \sum_{i=1}^d \operatorname{Area} (S^{d-1}\cap D_i)  \le \frac{d\alpha}{2\pi}\omega_{d-1},
  \]
which implies $\alpha > \frac{2\pi}{dn}$.  Now consider $a' =\frac{-1}d a$, the vector of norm $\frac1d$ in the direction opposite to $a$.  By applying the very colorful Carath\'eodory theorem as before, 
we can choose now $x_{d+1} \in X_{d+1}, x_{d+2} \in X_{d+2}, \ldots, x_{2d} \in X_{2d}$ such that
  \[
  a' \in \conv\{0, x_{d+1}, x_{d+2}, \ldots, x_{2d}\}.
  \]
  Now consider the set $K = \{x \in \frac1d S^{d-1} : \angle (x, a) \le \alpha\}$.  Let $x_1 \in X_1, \ldots, x_d \in X_d$ be the $d$ points that generate $C_1$.  Notice that the cone with apex $a'$ 
  and base $K$ is contained in $\conv\{x_1, x_2, \ldots, x_{2d}\}$.  Finding the radius $r$ of the largest ball around $0$ that is contained in this new cone is easily reduced to a $2$-dimensional 
  problem, giving
  \[
  r = \frac{\tan \alpha}{2d} > \frac{\alpha}{2d} > \frac{\pi}{nd^2} > (\pi/e^2)d^{-2d-2}.
  \]
  as we wanted.
\end{proof}

\section{Quantitative Helly theorems}\label{section-helly}

In this section we state precisely and give proofs for our Helly-type results.  As mentioned in the introduction, the first quantitative Helly-type theorem came from B\'ar\'any, Katchalski, and Pach's ground-breaking work \cite{baranykatchalskipach, MR750523}.  They measure the size of the intersection of a family of convex sets in two ways, using volume and diameter respectively.  Their result with diameter is essentially equivalent to that with volume, though the final constant obtained is slightly different.

%The quantity $n^*(d, \ep)$ is defined properly in Definition \ref{definition-for-helly}; its asymptotic growth is similar to that of $n(d,\ep)$:

Let us begin by showing that one can get a version of Helly's theorem with strong volumetric estimates.
\begin{theorem}[Quantitative Helly with volume]
\label{theorem-helly-for-volume}
%\label{approximatingvol}
Let $n=n^*(d, \ep)$ as in Definition \ref{definition-for-helly}.  Let $\ff$ be a finite family of convex sets such that for any subfamily $\ff'$ of at most $nd$ sets, $\vol\left( \bigcap \ff'\right) \ge 1 .$ Then, $\vol\left( \bigcap \ff \right) \ge (1+\ep)^{-1}.$
Moreover, if we only ask the condition for subfamilies of size $n-1$, the conclusion of the theorem may fail.
\end{theorem}

\begin{proof}[Proof of Theorem \ref{theorem-helly-for-volume}]
We may assume that $\bigcap \ff$ has non-empty interior. This was the first step in the original proof given in \cite{baranykatchalskipach}.  We may either use the same method or notice that if $n \ge 2$ we can actually use the``quantitative volume theorem'' of \cite[p.~109]{baranykatchalskipach} to obtain this. If $\bigcap \ff$ is not bounded, then it has infinite volume.  Moreover, we may assume that the sets in $\ff$ are closed half-spaces, or we could take the set of half-spaces containing $\bigcap \ff$ instead of $\ff$. Thus, it suffices to prove the following lemma.

\begin{lemma}
Let $\ff$ be a family of half-spaces such that $\bigcap \ff$ has volume $1$ and contains the origin in its interior. For $n=n^*(d,\ep)$, there is a subfamily $\ff' \subset \ff$ of at most $nd$ elements such that $\vol(\bigcap \ff') \le 1+\ep$.
\end{lemma}
%the intersection of half-spaces is always convex, so I removed this condition (David)

To prove the lemma, consider a polytope $P$ of $n$ facets containing $K=\bigcap \ff$ such that $\vol(P) \le 1+\ep$. Such a polytope exists by the definition of $n^*(d,\ep)$. After taking polars, we have $P^* \subset K^*$, and $P^*$ is a convex polytope with $n$ vertices.  Let $\ff^*$ be the family of polars of the elements in $\ff$.  Note that $\conv(\ff^*) = K^*$.  For each vertex $v$ of $P^*$, using the very colorful Carath\'eodory theorem we can find a set $A_v$ of $d$ points in $\ff^*$ such that $v \in \conv(A_v \cup \{0\})$.  Now consider $\ff' = (\cup_v A_v)^* \subset \ff$.  Notice that $\ff'$ is a subfamily of at most $nd$ elements.  Let us prove that $P^* \subset \conv[(\ff')^*]$.  By the definition of $\ff'$, it suffices to show that $0 \in \conv[(\ff')^*] $.  However, if that is not the case, there would be a closed half-space through $0$ containing $(\ff')^*$, and thus every vertex of $P^*$.  This would contradict the fact that $0$ is in the interior of $\conv(P)$.

Thus, we can find a subset $\ff' \subset \ff$ of at most $nd$ elements such that $P^* \subset \conv [(\ff')^*]$.  Then
	\[
	K \subset \bigcap \ff' \subset P,
	\]
	giving the desired result.
	
In order to prove optimality, let $K$ be a convex polytope of volume $1$ such that any polytope $P \supset K$ with at most $n-1$ facets has volume greater than $1+\ep$; this exists by the definition of $n^*(d,\ep)$.  Let $\mathcal{F}$ be the set of closed half-spaces that contain $K$ and define a facet of $K$.  Clearly, there is a $\delta >0$ such that the intersection of every $n-1$ elements of $\ff$ has volume at least $1+\ep+\delta$, but the intersection $\bigcap \ff$ is of volume $1$.
\end{proof}

Once we have constructed the polytope $P$, we can also finish the proof with the following folklore lemma that follows from Helly's theorem.

\begin{lemma}
Let $\mathcal{F}$ be a finite family of convex sets and $H$ a closed half-space such that $\bigcap \ff \subset H$ and $\bigcap \ff \neq \emptyset$.  Then, there is a subfamily $\ff' \subset \ff$ of at most $d$ sets such that $\bigcap \ff' \subset H$. 
\end{lemma}

Similar statements hold with other continuous functions, such as the diameter.  Given two convex sets $C, D$, we denote their Hausdorff distance $\delta_{H}(C,D)$; then, we have $|\operatorname{diam}(C)-\operatorname{diam}(D)| \le 2\delta_{H}(C,D).$  It is a classic problem to approximate a convex set by a polytope with few facets that contains it and is close in Hausdorff distance \cite{bronstein2008approximation, dudley1974metric, bronsteinivanov}. In particular, for any convex set $K$, there is a polytope $P \supset K$ with $O(\ep^{-(d-1)/2})$ facets which is at Hausdorff distance at most $\ep$ from $K$ (the $O$ notation hides constants depending on $K$).  Thus it makes sense to define the following.
%do we need the polytope to contain K? (David)

\begin{proposition}
\label{definition-diameter}
Let $d$ be a positive integer and $\ep>0$.  Then, there exists an integer $n$ such that for any convex set $K \subset \R^d$ with positive volume, there is a polytope $P \supset K$ of at most $n$ facets such that
\[
\operatorname{diam}(P) \le (1+\ep) \operatorname{diam}(K).
\]

We define $n^{\operatorname{diam}}(d,\ep)$ to be the smallest such value $n$.
\end{proposition}

\begin{proof}
From the discussion above, if we fix $K$, we know that $n^{\operatorname{diam}}(d,\ep,K)$ exists and is $O(\ep^{-(d-1)/2})$. Fix $\ep$ and $d$.

In order to get a universal bound for $n^{\operatorname{diam}}(d,\ep)$, note that it is sufficient to show the existence for the family $\mathcal{C}$ of closed convex sets $K \subset B_1 (0)$ with diameter $2$.  If there was no upper bound for $n^{\operatorname{diam}}(d,\ep)$, we would be able to find a sequence of convex sets such that $n^{\operatorname{diam}}(d,\ep,K_i) \to \infty$.  Since $\mathcal{C}$ is compact under the Hausdorff topology, there is a convergent subsequence.  If $K_i \to \tilde{K}$, one can see that polytopes that approximate $\tilde{K}$ very well would approximate $K_i$ as well if $i$ is large enough (a small perturbation is needed to fix containment, with arbitrarily small effect on the diameter).   This leads to the fact that $\limsup_{i \to \infty} n^{\operatorname{diam}}(d,\ep,K_i)$ is bounded by $n^{\operatorname{diam}}(d,\ep, \tilde{K})$, a contradiction.
\end{proof}

The fact that we have to work with convex sets up to homothetic copies is the reason why we can get bounds which approximate diameter with a relative error as opposed to an absolute error.  With $n^{\operatorname{diam}}(d,\ep)$ defined, we can state our quantitative Helly theorem for diameter.  The proof closely follows that of Theorem \ref{theorem-helly-for-volume}.

\begin{theorem}[Quantitative Helly with diameter]
\label{theorem-helly-for-diameter}
%\label{approximatingvol}
Let $n=n^{\operatorname{diam}}(d, \ep)$.  Let $\ff$ be a finite family of convex sets such that for any subfamily $\ff'$ of at most $nd$ sets, $\operatorname{diam}\left( \bigcap \ff'\right) \ge 1 .$ Then, $\operatorname{diam}\left( \bigcap \ff \right) \ge (1+\ep)^{-1}.$
Moreover, if we only ask the condition for subfamilies of size $n-1$, the conclusion of the theorem may fail.
\end{theorem}

%obtain part (b) of Theorem \ref{super-quant-helly}.  It is known that in order to approximate the unit sphere with a polytope within distance $\ep$ in the  Hausdorff metric, we require $\Omega(\ep^{-(d-1)/2})$ facets \cite{bronstein2008approximation}.  Thus $n^{\operatorname{diam}}(d, \ep)= \Omega(\ep^{-(d-1)/2})$.  Using this quantity, we immediately obtain the Theorem \ref{super-quant-helly}b.
We can also prove a colorful version of Theorem \ref{theorem-helly-for-volume}.

\begin{theorem}[Colorful quantitative Helly with volume]
\label{thm:hcolorcont}
For any positive integer $d$ and $\ep >0$, there exists $n=n^h(d, \ep)$ such that the following holds.  Let $\ff_1, \ldots, \ff_n$ be $n$ finite families of convex sets such that for every choice $K_1 \in \ff_1, \ldots, K_n \in \ff_n$ we have $ \vol{\left( \bigcap_{i=1}^n K_i\right)} \ge 1.$
Then, there is an index $i$ such that
\[
\vol\left( \bigcap \ff_i \right) \ge \frac{1}{1+\ep}.
\]
\end{theorem}

To prove Theorem \ref{thm:hcolorcont}, we will prove the following equivalent formulation.

\begin{theorem}
For any positive integer $d$ and $\ep >0$, there is an $n=n^h(d, \ep)$ such that the following holds.  Let $\ff_1, \ldots, \ff_n$ be $n$ families of closed half-spaces such that for each $i$ we have $ \vol\left( \bigcap_{H\in \ff_i} H\right) \le 1.$ 
Then, there is a choice $H_1 \in \ff_1, \ldots , H_n \in \ff_n$ such that
\[ \vol{\left( \bigcap_{i=1}^n H_i\right)} \le 1+\ep. \]
\end{theorem}

\begin{proof}
Let $\ep',\ep''$ be values depending on $\ep$, to be chosen later, and suppose that $n=\Omega(d\cdot n^*(d,\ep''))$.

Applying Theorem \ref{theorem-helly-for-volume} in the contrapositive, we replace each $\ff_i$ by a subset $\ff'_i\subseteq \ff_i$ such that we have $|\ff'_i|\le d\cdot n^*(d,\ep')$ and
\[ \vol\left( \bigcap_{H\in \ff'_i} H\right) \le 1+\ep'.\] We will assume that none of the hyperplanes in the $\ff'_i$ are parallel, else we could perturb them slightly.

\textbf{Claim.} There exists some choice of $H_1\in \ff'_1,\ldots,H_n\in \ff'_n$ such that $\bigcap_{i=1}^n H_i$ has finite volume.

Observe that translating nonparallel half-spaces in different directions does not affect whether their intersection has finite volume, though it may affect the value of that volume.  Given $H\in \ff'_i$, we may consider the hyperplane that defines this half-space; by invariance under translation, we may suppose that all these hyperplanes are tangent to the unit sphere centered at the origin.  Now, applying hyperplane-point duality, each family $\ff'_i$ is transformed to a family of points for which the convex hull contains the origin.  Applying standard colorful Helly's theorem, there must exist a rainbow set of points for which the convex hull contains the origin.  This corresponds to our desired choice $H_1\in \ff'_1,\ldots,H_n\in \ff'_n$, proving the claim.

Now, suppose that $H_1\in \ff'_1,\ldots,H_n\in \ff'_n$ are chosen such that $V=\vol\left( \bigcap_{i=1}^n H_i\right)$ attains the minimum value. Applying Theorem \ref{theorem-helly-for-volume}, again in the contrapositive, there exists some subset $S\subseteq \{1,\ldots,n\}$ such that $|S|\le d\cdot n^*(d,\ep'')$ such that $$\vol\left( \bigcap_{i\in S} H_i\right)\le (1+\ep'')V.$$Let $P$ be the polytope defined by $\{H_i\mid i\in S\}$, and let $j$ be an element of $\{1,\ldots,n\}\backslash S$.  We will attempt to find some $H\in \ff'_j$ that significantly reduces the volume of $P$.

Suppose towards a contradiction that, for each $H$, we have $$\vol(P\cap H)> (1+\ep'')V-\frac{1}{|\ff'_j|}\left[(1+\ep'')V-(1+\ep')\right].$$Then, we would have $$\vol\left( \bigcap_{H\in \ff'_j} H\right)\ge \bigcap_{H\in \ff'_j} \vol(P\cap H)>1+\ep',$$a contradiction. Hence, for some $H$ we must have
\begin{align*}
\vol(P\cap H)&\le (1+\ep'')V-\frac{1}{|\ff'_j|}\left[(1+\ep'')V-(1+\ep')\right]\\
&\le (1+\ep'')V-\frac{(1+\ep'')V-(1+\ep')}{d\cdot n^*(d,\ep')}.
\end{align*}

However, we assumed that the intersection of any colorful set of half-spaces has volume at least $V$.  Hence, $$V\le (1+\ep'')V-\frac{(1+\ep'')V-(1+\ep')}{d\cdot n^*(d,\ep')},$$which rearranges to $$V\le \frac{1+\ep'}{(1+\ep'')-\ep''\cdot d\cdot n^*(d,\ep')}\approx 1+\ep'\cdot \ep''\cdot d\cdot n^*(d,\ep').$$Thus, the theorem holds if we choose $\ep',\ep''$ such that $\ep''\cdot d\cdot n^*(d,\ep')\ll 1$ and $\ep'\cdot \ep''\cdot d\cdot n^*(d,\ep')<\ep$.
\end{proof}

A different proof is shown in \cite{Sob15}.

%
%\begin{theorem}[Quantitative Helly with diameter]
%\label{theorem-helly-for-diameter}
%	Let $n=n^{\operatorname{diam}}(d, \ep)$ as in Proposition \ref{definition-diameter}.  Let $\ff$ be a finite family of convex sets such that for any subfamily $\ff'$ of at most $nd$ sets, 
%\[
%\operatorname{diam}\left( \bigcap \ff'\right) \ge 1	.
%\]
%Then,
%\[
%\operatorname{diam}\left( \bigcap \ff \right) \ge (1+\ep)^{-1}.
%\]
%Moreover, $n$ is a lower bound for the size of the subfamilies $\ff'$ that we need to check.
%\end{theorem}

\section{Quantitative Tverberg theorems}\label{section-Tverberg}

We begin with a version for Tverberg's theorem that requires the intersection of the convex hulls of the parts to include a ball of large radius.

\begin{theorem}[Continuous quantitative Tverberg] \label{theorem-tverberg-volume}
Let $d, m$ be positive integers, $n=(2dm-1)(d+1)+1$ and $T_1, T_2, \ldots, T_n$ be subsets of $\R^d$ such that the convex hull of each $T_i$ contains a Euclidean unit ball, $B_1(c_i)$.  
Then, we can choose points $t_1 \in T_1, t_2 \in T_2, \ldots, t_n \in T_n$ and a partition of $\{ t_1, t_2, \ldots, t_n\}$ into $m$ sets $A_1, A_2, \ldots, A_m$ such that the intersection
\[
\bigcap_{i=1}^m \conv{A_i}
\]
%Can consolidate this to one line to save space (David)
contains a ball of radius $(\pi/e^2)d^{-2d-2}$.

Moreover, if  we take $n'= n^{\operatorname{bms}}(d, \ep)$ as in Definition \ref{banach-mazur-value-smooth} and instead of the $n$ above let
\[
n= (m [(n'-1)d+1]-1)(d+1)+1,
\]
then we can guarantee that $\bigcap_{i=1}^m \conv A_i$ contains a ball of radius $(1+\ep)^{-1}$.
\end{theorem}

\begin{proof}[Proof of Theorem \ref{theorem-tverberg-volume}]
First consider the case with $(2dm-1)(d+1)+1$ sets.   Consider $C = \{c_1, c_2, \ldots, c_n\}$ the set of centers of the balls of unit radius defined in the statement of the theorem.  If we use the standard Tverberg's theorem with the set $C$, we can find a partition of $C$ into $2dm$ sets $C_1, C_2, \ldots, C_{2dm}$ such that their convex hulls intersect in some point $p$.
	
Now we split these $2dm$ parts into $m$ blocks of $2d$ parts each in an arbitrary way.  We now show that in each block, we can pick one point of each of its corresponding $T_i$ such that the convex hull of the resulting set contains $B_{r(d)}(p)$, where $r(d)$ is the constant in the quantitative Steinitz theorem.  This effectively proves the theorem, since we have $r(d) \ge (\pi/e^2)d^{-2d-2}$.

Without loss of generality, we assume that one such block is $C_1, C_2, \ldots, C_{2d}$.  For each $C_i$ consider
	\[
	\tilde{C}_i = \bigcup\{T_j : c_j \in C_i\}.
	\]
	Since $\conv(T_j) \supset c_j + B_1(0)$, we have
	\[
	\conv(\tilde{C}_i) \supset \conv(C_i) + B_1(0) \supset p + B_1(0) = B_1(p).
	\]
	Thus, we can apply Lemma \ref{theorem-colored-steinitz}, our colorful Steinitz with guaranteed containment of small balls,  to the sets $\tilde{C}_1, \tilde{C}_2, \ldots, \tilde{C}_{2d}$ and obtain a set $\{t_1, t_2, \ldots, t_{2d}\}$ with $t_i \in \tilde{C}_i$ whose convex hull contains $B_{r(d)}(p)$, as desired.
	If we are given instead $n= (m\cdot [(n'-1)d+1]-1)(d+1)+1$, we can split the sets of balls' centers into $m$ blocks of size $(n'-1)d+1$, which allows us to use Proposition \ref{corollary-steinitz-for-balls} to reach the conclusion. 
\end{proof}

\noindent \textbf{Remark.}  The resulting set of the proof above uses only $2dm$ points, so most sets $T_i$ are not being used at all; this suggests that a stronger statement may hold. Moreover, once we get the first Tverberg partition, we have complete freedom on how to split the $2dm$ parts into $m$ blocks of equal size. Thus, our approach in fact shows that there exist $\sim m^{2dm}$ different Tverberg partitions of this kind.

For an integer $q$, let $\lceil q \rceil_p$ be the smallest prime which is greater than or equal to $q$.  Then

\begin{theorem}[Colorful continuous quantitative Tverberg] \label{theorem-colorful-tverberg-volume}

Let $m, d$ be positive integers, $n=\lceil 2md+1\rceil_p -1$ and $F_1, F_2, \ldots, F_{d+1}$ be families of $n$ sets of points of $\R^d$ each.  We consider the families $F_i = \{T_{i,j}: 1 \le j \le n\}$ as the color classes.  Suppose that $\conv (T_{i,j})$ contains a ball of radius $1$ for all $i,j$.  Then, there is a choice of points $t_{i,j} \in T_{i,j}$ and a partition of the resulting set into $m$ parts $A_1, \ldots, A_{m}$ such that each $A_i$ contains at most $2d$ points of each color class and $\bigcap_{i=1}^m \conv (A_i)$ contains a ball of radius $(\pi/e^2)d^{-2d-2)}$.
	
In addition, if we take instead $n'= n^{\operatorname{bms}}(d,\ep)$ and 
\[
n=\lceil m\cdot ((n'-1)d+1)+1\rceil_p -1,
\]
and allow each $A_i$ to have $(n'-1)d+1$ points of each color, then in the conclusion we can guarantee that $\bigcap_{i=1}^m \conv (A_i)$ contains a ball of radius $(1+\ep)^{-1}$.
\end{theorem}

The reason why we require the use of $\lceil 2dm \rceil_p$ is the conditions for the known cases of Conjecture \ref{conjecture-colorful-Tverberg}.  If Conjecture \ref{conjecture-colorful-Tverberg} were proved, we could use $2dm$ sets in each color class instead. However, since the prime number theorem implies $\lim_{q \to \infty} \frac{\lceil q \rceil_p}{q} =1$ and in the small cases we have $\lceil q \rceil_p < 2q$, the result above is almost as good. We should note that the ``optimal colorful Tverberg'' by Blagojevi\'c, Matschke, and Ziegler \cite[Theorem 2.1]{bmz-optimal} also admits a volumetric version as above, with essentially the same proof. 

If all $T_{i,j}$ are equal to $B_1(0)$, the need to allow each $A_i$ to have more points from each color class becomes apparent from the results of inaproximability of the sphere by polytopes with few vertices \cite{bronstein2008approximation}.  The condition we have is saying that the number of points from $F_j$ in $A_i$ should not exceed $\frac{1}{m}|F_j|$.  We know that a subset of $B_1(0)$ that contains $B_{1-\ep}(0)$ should have at least $n^{\operatorname{bms}}(d,\ep)$ points, showing that the number of points we are allowing to take from each color class is optimal up to a multiplicative factor of $\sim d^2$.

\begin{proof}[Proof of Theorem \ref{theorem-colorful-tverberg-volume}]
	Let $c_{i,j}$ be the center of a ball of unit radius contained in $\conv (T_{i,j})$.  Note that we can apply the colorful Tverberg theorem in \cite[Theorem 2.1]{bmz-optimal} to the set of centers to obtain a colorful partition of them into $\lceil 2dm +1\rceil_p-1 \ge 2dm$ sets whose convex hulls intersect.  As in the proof of Theorem \ref{theorem-tverberg-volume}, we may split these sets into $m$ blocks of exactly $2d$ parts each, leaving perhaps some sets unused.
	The same application of Lemma \ref{theorem-colored-steinitz} gives us the desired result.  If we seek a ball of almost the same radius in the end, Theorem \ref{corollary-steinitz-for-balls} completes the proof.
\end{proof}

\section*{Acknowledgments} We are grateful to I.~B\'ar\'any, A.~Barvinok, F.~Frick, A.~Holmsen, J.~Pach, and G.M.~Ziegler for their comments and suggestions. This work was partially supported by the Institute for Mathematics and its Applications (IMA) in Minneapolis, MN funded by the National Science Foundation (NSF). The authors are grateful for the wonderful working environment that led to this paper. The research of De Loera and La Haye was also supported by a UC MEXUS grant. Rolnick was additionally supported by NSF grants DMS-1321794 and 1122374.

\bibliographystyle{plain}

\bibliography{references}

\noindent J.A. De Loera and R.N. La Haye\\
\textsc{
Department of Mathematics \\
University of California, Davis \\
Davis, CA 95616
 \\
}\\[0.1cm]
\noindent D. Rolnick \\
\textsc{
 Department of Mathematics \\
 Massachusetts Institute of Technology \\
 Cambridge, MA 02139
 \\
}\\[0.1cm]
\noindent P. Sober\'on \\
\textsc{
Mathematics Department \\
Northeastern University \\
Boston, MA 02115
}\\[0.1cm]

\noindent \textit{E-mail addresses: }\texttt{deloera@math.ucdavis.edu, rlahaye@math.ucdavis.edu, drolnick@math.mit.edu, p.soberonbravo@neu.edu}
\end{document}